\documentclass[a4paper,10pt,oneside, reqno]{amsart}

\usepackage{amssymb}
\usepackage{amsmath}
\usepackage{amsthm}
\usepackage[fleqn]{mathtools}
\usepackage{amsfonts}
\usepackage{stmaryrd}
\usepackage[english]{babel}
\usepackage[utf8]{inputenc}
\usepackage{enumerate}
\usepackage{graphicx,pinlabel}

\theoremstyle{plain}
\newtheorem{theorem}{Theorem}[section]
\newtheorem{corollary}[theorem]{Corollary}
\newtheorem{lemma}[theorem]{Lemma}
\newtheorem{proposition}[theorem]{Proposition}

\newtheorem{theorem*}{Theorem}
\newtheorem{corollary*}{Corollary}

\theoremstyle{definition}
\newtheorem{definition}[theorem]{Definition}

\newcommand{\ST}{\operatorname{{\mathcal T}}}
\newcommand{\Z}{\operatorname{{\mathbb Z}}}

\begin{document}

\title{Spectral properties of the Cayley Graphs of \\ split metacyclic groups}
\author[Kashyap Rajeevsarathy]{K.Rajeevsarathy}
\address{Department of Mathematics\\
Indian Institute of Science Education and Research Bhopal\\
Bhopal Bypass Road, Bhauri \\
Bhopal 462 066, Madhya Pradesh\\
India}
\email{kashyap@iiserb.ac.in}
\urladdr{https://home.iiserb.ac.in/$_{\widetilde{\phantom{n}}}$kashyap/}

\author[Siddhartha Sarkar]{S. Sarkar}
\address{Department of Mathematics\\
Indian Institute of Science Education and Research Bhopal\\
Bhopal Bypass Road, Bhauri \\
Bhopal 462 066, Madhya Pradesh\\
India}
\email{sidhu@iiserb.ac.in}
\urladdr{https://home.iiserb.ac.in/$_{\widetilde{\phantom{n}}}$sidhu/}

\author[Sivaramakrishnan Lakshmivarahan]{S. Lakshmivarahan}
\address{The University of Oklahoma\\
School of Computer Science\\
110 W. Boyd St.\\
Devon Energy Hall, Rm. 230\\
Norman, OK 73019\\
USA}
\email{varahan@ou.edu}
\urladdr{http://www.ou.edu/content/coe/cs/people/varahan.html}

\author[Pawan Kumar Aurora]{P. K. Aurora}
\address{Department of Electrical Engineering and Computer Science\\
Indian Institute of Science Education and Research Bhopal\\
Bhopal Bypass Road, Bhauri \\
Bhopal 462 066, Madhya Pradesh\\
India}
\email{paurora@iiserb.ac.in}

\subjclass[2010]{Primary 68R10; Secondary 05C50}


\keywords{Cayley Graphs, Ramanujan Graphs, Cyclic groups, Semi-direct product}

\begin{abstract}
Let $\Gamma(G,S)$ denote the Cayley graph of a group $G$ with respect to a set $S \subset G$. In this paper, we analyze the spectral properties of the Cayley graphs $\ST_{m,n,k} = \Gamma(\mathbb{Z}_m \ltimes_k \mathbb{Z}_n, \{(\pm 1,0),(0,\pm 1)\})$, where $m,n \geq 3$ and $k^m \equiv 1 \pmod{n}$. We show that the adjacency matrix of $\ST_{m,n,k}$, upto relabeling, is a block circulant matrix, and we also obtain an explicit description of these blocks. By extending a result due to Walker-Mieghem to Hermitian matrices, we show that $\ST_{m,n,k}$ is not Ramanujan, when either $m > 8$, or $n \geq 400$. 
\end{abstract}

\maketitle

\section{Introduction}

For a finite group $G$ and a subset $S$ of $G$, let $\Gamma(G,S)$ denote the Cayley Graph of $G$ with respect to the set $S$. In this paper, we analyze the spectral properties of a collection of connected undirected Cayley graphs of degree $4$ arising from metacyclic groups, namely the graphs $\{\ST_{m,n,k}\}$, where  $$\ST_{m,n,k} = \Gamma(\mathbb{Z}_m \ltimes_k \mathbb{Z}_n, \{(\pm1,0),(0,\pm 1)\}), \, m,n \geq 3 \text{ and } k^m \equiv 1 \pmod{n}.$$ Our analysis is mainly driven by the fact that the properties of a subcollection of graphs known as \textit{supertoroids}, have been widely studied~\cite{DR2,DR1,LV2,VD1} clasically, and more recently, the diamater~\cite{RS} and the degree-diameter problem of this family of graphs have also been analyzed (see~\cite{E1,MH1,VT1}). 

 Let $A(X)$ denote the adjacency matrix of a graph $X$, let $\Theta$ denote the zero matrix of appropriate dimension, and let $\text{Diag}_m(p_1, p_2, \dotsc, p_{\tau})$ denote the diagonal matrix $(a_{ij})_{m \times m}$ defined by $a_{ii} = 1$, if $i \in \{p_1,\ldots,p_\tau\}$, and $a_{ii} = 0$, otherwise. By using a special labeling scheme for the vertices of $\ST_{m,n,k}$ and applying tools from elementary number theory, we obtain the main result of the paper.
 
\begin{theorem*}
Let $\alpha$ be the order of $k$ in $\mathbb{Z}_n^{\times}$, and let $C_m$ denote the $m$-cycle graph. For $0 \leq i,j \leq n-1$, the matrix $A(\ST_{m,n,k}) = (A_{ij})$ is block-circulant with $m \times m$ blocks such that for each $i$, $A_{ii} = A(C_m)$ and for each $j \neq i$, $A_{ij}$ is a diagonal matrix with $0$ or $t$, or $2t$ nonzero entries given by the following conditions. 
\begin{enumerate}[(i)]
\item If $k^{s} \not\equiv -1 \pmod{n}$, for any $s$, then
\[
A_{ij} = 
\begin{cases}
{\mathrm{Diag}}(\xi, \xi+\alpha, \dotsc, \xi+(t-1)\alpha),  &\mbox{if } j \equiv \pm k^{\alpha - \xi}, \text{ for } 0 \leq \xi \leq \alpha - 1, \text{ and } \\
\Theta, &\mbox{otherwise. } 
\end{cases}
\]
\item  If $\alpha$ is even and $k^{\alpha/2} \equiv -1 \pmod{n}$, then 
\[
A_{ij} = 
\begin{cases}
{\mathrm{Diag}}(\xi, \xi+{\frac {\alpha}{2}}, \dotsc, \xi+(2t-1){\frac {\alpha}{2}}),  &\mbox{if } j \equiv  \pm k^{\alpha - \xi}, \text{ for }0 \leq \xi \leq (\alpha/2) - 1, \text{ and } \\
\Theta, &\mbox{otherwise.} 
\end{cases}
\]
\end{enumerate}
\end{theorem*}

\noindent By appealing to the theory of circulant matrices~\cite{PD1}, we show that the matrix $A(\ST_{m,n,k})$ is conjugate with a block diagonal matrix $M = (M_{ij})_{n\times n}$ with Hermitian $m \times m$ blocks, one of whose blocks has the eigenvalues $2+2\cos(2\pi k/m)$, $0 \leq k \leq m-1$. Extending a result due to Walker-Mieghem~\cite{WM1} that bounds the largest eigenvalue of a symmetric matrix, to the case of Hermitian matrices (see Corollary~\ref{cor:lambdaX_bound}), we derive a lower bound on the largest eigenvalue of $M$. As a consequence of this spectral analysis of $A(\ST_{m,n,k})$, we determine lower bounds on $m$ and $n$ for the nonexistence of Ramanujan graphs~\cite{RM1} in the family $\{\ST_{m,n,k}\}$. 
\begin{corollary*}
The graph $\ST_{m,n,k}$ is not Ramanujan when either $m > 8$, or $n \geq 400$. 
\end{corollary*}
\noindent Finally, using software written for Mathematica~\cite{MAT}, we list the collection of all pairs in $\{(m,n) : 3 \leq m \leq 8 \text{ and } 3 \leq n < 400\}$ for which there exist at least one $k >1$ such that $\ST_{m,n,k}$ is Ramanujan.

This paper is organized in the following manner. In Section 2, we provide a complete analysis of the structure of $A(\ST_{m,n,k})$ and the main result of this paper, while the spectral properties of $\ST_{m,n,k}$ are analyzed in Section 3.

\section{Adjacency matrix of $\ST_{m,n,k}$}
In this section, we obtain a complete description of the structure of $A(\ST_{m,n,k})$. For integers $m,n \geq 3$, let the group $\Z_{m} \ltimes_k \Z_n$ be given by the presentation
\[
\Z_{m} \ltimes_k \Z_n = \langle x, y : x^m = 1 = y^n, \, x^{-1} yx = y^k \rangle,
\]
where $k \in {\mathbb Z}_n^{\times}$ is of order $\alpha > 1$ satisfying $k^m \equiv 1 \pmod{n}$. Clearly, $m = t \alpha$, and we will call $t$ as the {\textit period} of $k$. Denoting the vertex set of a graph $X$ by $V(X)$ and its edge set by $E(X)$, we have the following.
\begin{definition}
 For $i \in {\mathbb Z}_n$ we define the \textit{$i^{th}$ packet} $V_i$ of $V(\ST_{m,n,k})$ as
\[
V_i := \{ x^{\tau \alpha + \xi} y^{i k^{\xi}} : 0 \leq \tau \leq t-1 \text{ and } 0 \leq \xi \leq \alpha -1 \}.
\] 
\end{definition}
\noindent By definition, it is clear that $|V_i| = t \alpha = m$, and $V(\ST_{m,n,k}) = \sqcup_{i=0}^{n-1} V_i$. We will now define an an ordering on each packet. 
\begin{definition}
\label{def:pack-order}
Given $x^{\tau \alpha + \xi} y^{i k^{\xi}}, x^{\tau^{\prime} \alpha + \xi^{\prime}} y^{i k^{\xi^{\prime}}} \in V_i$, we define 
\[
x^{\tau \alpha + \xi} y^{i k^{\xi}} < x^{\tau^{\prime} \alpha + \xi^{\prime}} y^{i k^{\xi^{\prime}}} \text{if, and only if,} \text{ either }\tau < \tau^{\prime}, \text{ or } \tau = \tau^{\prime} \text{ and } \xi < \xi^{\prime}.
\]
\end{definition} 
\noindent We will now define an ordering on the vertices across the $V_i$. 
\begin{definition}
\label{def:vert-order}
 Given $v_1 \in V_i$ and $v_2 \in V_j$, for $i \neq j$, we define 
 $$v_1 < v_2 \text{ if, and only if, }i < j.$$ 
\end{definition}
\noindent The orderings in Definitions~\ref{def:pack-order} and~\ref{def:vert-order} together yield an ordering on $V(\ST_{m,n,k})$. Hence, we have $V(\ST_{m,n,k}) = \{v_0,v_1, \ldots,v_{mn-1}\}$, where $v_i < v_j$, for $i < j$, and let $A(\ST_{m,n,k})$ be the adjacency matrix associated with this labeling. It is apparent that this ordering induces a block matrix structure on $A(\ST_{m,n,k})$ 
\[
  A(\ST_{m,n,k}) =
   \begin{pmatrix}
   A_{00} & A_{01} & \cdots & A_{0\,n-1} \\
   A_{10} & A_{11} & \cdots & A_{1 \,n-1} \\
   \vdots & \vdots & \ddots & \vdots \\
   A_{n-1 \, 0}    & A_{n-1 \, 1}  & \cdots & A_{n-1\,n-1}  \\
  \end{pmatrix},
\]  
where each block $A_{ij}$ is an $m \times m$ matrix that represents the adjacency between the vertex packets $V_i$ and $V_j$.

\begin{definition} An edge $(v_1, v_2) \in E(\ST_{m,n,k})$ is called an \textit{$x$-sibling pair} (resp. \textit{$y$-sibling pair}) if $v_2 = v_1 x^{\pm 1}$ (resp. $v_2 = v_1 y^{\pm 1}$). 
\end{definition}

\begin{lemma}
\label{lem:xy-sibling}
Let $(v_1,v_2) \in V_i \times V_j$, for $i \neq j$. Then:
\begin{enumerate}[(i)]
\item $(v_1, v_2)$ is an $x$-sibling pair if, and only if, $i \equiv j \pmod{n}$, and
\item $(v_1, v_2)$ is an $y$-sibling pair implies $i \not\equiv j \pmod{n}$.
\end{enumerate}
\end{lemma}

\begin{proof} Let $v_1 = x^{\tau \alpha + \xi} y^{i k^{\xi}}$ and $v_2 = x^{\tau^{\prime} \alpha + \xi^{\prime}} y^{j k^{\xi^{\prime}}}$ where $0 \leq \tau, \tau^{\prime} \leq t - 1$ and $0 \leq \xi, \xi^{\prime} \leq \alpha -1$. Then $v_2 = v_1 x^{\pm 1}$ if, and only if,
\[
\tau^{\prime} \alpha + \xi^{\prime} \equiv \tau \alpha + \xi \pm 1 \pmod{m} \text{ and } j k^{\xi^{\prime}} \equiv i k^{\xi \pm 1} ~({\mathrm {mod}}~n), 
\]
and (i) follows.

To show (ii), we assume on the contrary that $i \equiv j\pmod{n}$. Since by hypothesis $j k^{\xi^{\prime}} \equiv i k^{\xi} \pm 1 \pmod{n}$, our assumption would imply that $1 \equiv 0 \pmod{n}$, which is impossible.
\end{proof}

\begin{corollary}
\label{cor:circ_diag_blocks}
For $0 \leq i\leq n-1$, the sub-blocks $A_{ii}$ of $A(\ST_{m,n,k})$ is a circulant matrix. Moreover, $A_{ii} = A_{jj}$, whenever $i \neq j$.
\end{corollary}
\begin{proof}
 From Lemma~\ref{lem:xy-sibling}, it is apparent that the nonzero entries in $A(\ST_{m,n,k})$ contributed by the edges in $E(\ST_{m,n,k})$ forming $x$-sibling pairs appear in the diagonal blocks $A_{ii}$, while those contributed by the $y$-sibling pairs appear in the off diagonal blocks $A_{ij}$, for $i \neq j$.  Given a vertex $v=  x^{\tau \alpha + \xi} y^{i k^{\xi}} \in V_i$, since $vx = x^{\tau \alpha + \xi + 1} y^{ik^{\xi + 1}}$, the ordering defined on $V(\ST_{m,n,k})$ makes $vx$ an immediate successor of $v$, except when $(\tau, \alpha) = (t-1, \alpha-1)$, in which case $v$ and $vx$ (resp.) represent the last and first vertices (resp.) of $V_i$ (under the ordering). Hence, for each $i$, we have $A_{ii} = (b_{rs})_{m\times m}$, where $b_{rs} = 1$, if $|r - s| > 1$, and $b_{rs} = 0$, otherwise.  
\end{proof}

Each diagonal block $A_{ii}$ of $A(\ST_{m,n,k})$ described in Corollary~\ref{cor:circ_diag_blocks} is, in fact, the adjacency matrix $A(C_m)$ of the $m$-cycle graph $C_m$.  We will now turn our attention to the $A_{ij}$, for $i \neq j$. In this regard, it would be convenient to further partition each packet $V_i = \sqcup_{\tau = 0}^{t-1} V_{i, \tau}$ into \textit{$(i, \tau)$-subpackets}, where 
$V_{i, \tau} := \{ x^{\tau \alpha + \xi} y^{i k^{\xi}} : 0 \leq \xi \leq \alpha -1\}$. This partition induces a block diagonal structure on each off-diagonal block $A_{ij}$ of the form
\[\tag{*}
  A_{ij} =
   \begin{pmatrix}
   B^{ij}_{00} & B^{ij}_{01} & \cdots & B^{ij}_{0,t-1} \\
   B^{ij}_{10} & B^{ij}_{11} & \cdots & B^{ij}_{1,t-1} \\
   \vdots & \vdots & \ddots & \vdots \\
   B^{ij}_{t-1,0}    & B^{ij}_{t-1,1}  & \cdots & B^{ij}_{t-1,t-1}  \\
  \end{pmatrix}
\]  
where $B^{ij}_{\tau, \mu}$ arise from the adjacencies between the vertices in the packets $V_{i, \tau}, V_{j,\mu}$. The leads us to the following lemma, which describes the structure of the blocks $B^{ij}_{\tau, \mu}$.

\begin{lemma}
\label{lem:Aij_subblocks}
Consider the structure of each off-diagonal block $A_{ij}$  of $A(\ST_{m,n,k})$, as described in (*) above. Then:
\begin{enumerate}[(i)]
\item $B_{\tau, \mu}^{ij}$ is a null matrix, if $\tau \neq \mu$, and
\item $B^{ij}_{\tau, \tau} = B^{ij}_{\mu, \mu}$, for each $0 \leq \tau, \mu \leq t-1$.
\end{enumerate}
\end{lemma}

\begin{proof} As $i \not\equiv j \pmod{n}$, the nonzero entries of the matrix $A_{ij}$ correspond to the edges that form $y$-sibling pairs. Now if $v_1 = x^{\tau \alpha + \xi} y^{i k^{\xi}}, \, v_2 = x^{\mu \alpha + \xi^{\prime}} y^{j k^{\xi^{\prime}}}$, then the corresponding matrix entry is $1$ if, and only if, $v_2 = v_1 y^{\pm 1}$. This implies $\tau = \mu$ and $\xi = \xi^{\prime}$, and (i) follows. 

For (ii), assume without loss of generality, that $\tau > \mu$. Then $(x^{\tau \alpha + \xi} y^{i k^{\xi}}, x^{\tau \alpha + \xi^{\prime}} y^{j k^{\xi^{\prime}}}) \in E(\ST_{m,n,k})$, if, and only if, $(x^{\mu \alpha + \xi} y^{i k^{\xi}}, x^{\mu \alpha + \xi^{\prime}} y^{j k^{\xi^{\prime}}}) \in E(\ST_{m,n,k})$, and the result follows.  
\end{proof}

\noindent By Lemma~\ref{lem:Aij_subblocks}, we can see that when $i \neq j$, $A_{ij}$ has the form
\[
  A_{ij} =
   \begin{pmatrix}
   B^{ij}_{00} & \Theta & \cdots & \Theta & \Theta \\
   \Theta & B^{ij}_{00} & \cdots & \Theta & \Theta \\
   \vdots & \vdots & \ddots & \vdots \\
   \Theta    & \Theta  & \cdots & B^{ij}_{00} & \Theta \\
   \Theta    & \Theta  & \cdots & \Theta & B^{ij}_{00}  \\
  \end{pmatrix}
\]  
where $\Theta$ represents the null matrix, and so it remains to obtain a description of $B^{ij}_{00}$. We begin by noting that $(x^{\tau \alpha + \xi} y^{i k^{\xi}}, x^{\tau \alpha + \xi^{\prime}} y^{j k^{\xi^{\prime}}}) \in E(\ST_{m,n,k})$ if, and only if, 
\[
\xi \equiv \xi^{\prime} \pmod{\alpha} \text{ and } jk^{\xi^{\prime}} \equiv ik^{\xi} \pm 1 \pmod{n},
\]
which further holds true if, and only if,
\begin{equation}
\label{eqn:ij_condn}
j \equiv i \pm k^{\alpha - \xi} \pmod{n}.
\end{equation}
The condition $\xi \equiv \xi^{\prime} ({\mathrm {mod}}~ \alpha)$ shows that the matrix $B^{ij}_{00}$ is a diagonal matrix for each $\tau$. Hence, we have the following. 

\begin{lemma}
$B^{ij}_{00}$ is a diagonal matrix whose nonzero entries of $B^{ij}_{00}$ are contributed by the edges $(x^{\tau \alpha + \xi} y^{ik^{\xi}}, x^{\tau \alpha + \xi} y^{jk^{\xi}})$ and $(x^{\tau \alpha + \xi^{\prime}} y^{ik^{\xi^{\prime}}}, x^{\tau \alpha + \xi^{\prime}} y^{jk^{\xi^{\prime}}})$, for $0 \leq \xi, \xi^{\prime} \leq \alpha - 1$, satisfying the conditions
\begin{equation}
\label{eqn:ij_condn2}
j \equiv i + k^{\alpha - \xi} \text{ and } j \equiv i - k^{\alpha - \xi^{\prime}}. 
\end{equation}
In particular, if $i, j \in {\mathbb Z}_n$ such that neither $i - j$ nor $j - i$ is a power of the unit $k$, then $A_{ij}$ is a null matrix. 
\end{lemma}

\noindent In the following lemma, we further analyze the $B^{ij}_{00}$ to show that it can have at most two nonzero entries.  

\begin{lemma}
\label{lem:B00ij_diag_structure}
Consider $i,j, \in \mathbb{Z}_n$ satisfying equation~(\ref{eqn:ij_condn}).
\begin{enumerate}[(i)]
\item If $k^{\eta} \not\equiv -1 \pmod{n}$ for any $\eta \in \{ 0, \dotsc, \alpha - 1 \}$, then $B^{ij}_{00}$ has exactly one nonzero entry. In particular, if $\alpha$ is odd, then $B^{ij}_{00}$ exactly one nonzero entry.
\item if $k^{\eta} \equiv -1 \pmod{n}$, for some $\eta \in \{ 0, \dotsc, \alpha - 1 \}$, then $\alpha$ is even and $B^{ij}_{00}$ has exactly two nonzero entries contributed by the edges $(x^{\tau \alpha + \xi} y^{ik^{\xi}}, x^{\tau \alpha + \xi} y^{jk^{\xi}})$ and $(x^{\tau \alpha + \xi + \alpha/2} y^{ik^{\xi + \alpha/2}}, x^{\tau \alpha + \xi + \alpha/2} y^{jk^{\xi + \alpha/2}})$.
\end{enumerate}
\end{lemma}

\begin{proof} Clearly, when $k^{\eta} \not\equiv -1 \pmod{n}$, there is exactly one solution to equation~(\ref{eqn:ij_condn}), and so (i) follows.  For (ii), it suffices to show that $\alpha$ is even. To see this, we  first note that if equation~\ref{eqn:ij_condn2} holds, then $\xi \not\equiv \xi^{\prime} \pmod{\alpha}$. Assuming without loss of generality that $\xi < \xi^{\prime}$, this would imply $k^{\xi^{\prime} - \xi} + 1 \equiv 0 \pmod{n}$, and so  $\alpha \mid 2(\xi^{\prime} - \xi)$. If $\alpha$ is odd, then $\alpha \mid (\xi^{\prime} - \xi)$, which would contradict the fact that $\xi \not\equiv \xi^{\prime} \pmod{\alpha}$. Thus, we have $\alpha$ is even, and  $\xi^{\prime} - \xi \equiv \alpha/2 \pmod{\alpha}$, and (ii) follows.
\end{proof}

 Let $\text{Diag}_m(p_1, p_2, \dotsc, p_{\tau})$ denote the diagonal matrix $(a_{ij})_{m \times m}$ defined by $a_{ii} = 1$, if $i \in \{p_1,\ldots,p_\tau\}$, and $a_{ii} = 0$, otherwise. Putting together the results in Corollary~\ref{cor:circ_diag_blocks}, and Lemmas~\ref{lem:Aij_subblocks} -~\ref{lem:B00ij_diag_structure}, and the fact that $A_{ij} = A_{i+1, j+1}$, we have the main the result of the paper, which gives a complete description of $A(\ST_{m,n,k})$.

\begin{theorem}[Main result] 
\label{thm:main}
The matrix $A(\ST_{m,n,k}) = (A_{ij})_{n \times n}$ is block-circulant with $m \times m$ blocks such that for each $i$, $A_{ii} = A(C_m)$ and for each $j \neq i$, $A_{ij}$ is a diagonal matrix with $0$ or $t$, or $2t$ nonzero entries given by the following conditions. 
\begin{enumerate}[(i)]
\item If $k^{s} \not\equiv -1 \pmod{n}$, for any $s$, then
\[
A_{ij} = 
\begin{cases}
{\mathrm{Diag}}(\xi, \xi+\alpha, \dotsc, \xi+(t-1)\alpha),  &\mbox{if } j \equiv \pm k^{\alpha - \xi}, \text{ for } 0 \leq \xi \leq \alpha - 1, \text{ and } \\
\Theta, &\mbox{otherwise. } 
\end{cases}
\]
\item  If $\alpha$ is even and $k^{\alpha/2} \equiv -1 \pmod{n}$, then 
\[
A_{ij} = 
\begin{cases}
{\mathrm{Diag}}(\xi, \xi+{\frac {\alpha}{2}}, \dotsc, \xi+(2t-1){\frac {\alpha}{2}}),  &\mbox{if } j \equiv  \pm k^{\alpha - \xi}, \text{ for }0 \leq \xi \leq (\alpha/2) - 1, \text{ and } \\
\Theta, &\mbox{otherwise.} 
\end{cases}
\]
\end{enumerate}
\end{theorem}

\section{On the spectrum of $\ST_{m,n,k}$} The central goal of this section is to analyze the spectrum of $A(\ST_{m,n,k}$).  From Theorem~\ref{thm:main}, we know that all succsessive rows of blocks of $A(\ST_{m,n,k}) = (A_{ij})$ are obtained by cyclically permuting the first row of blocks  $(A_{00}, \ldots, A_{0\,m-1})$. It is well known~\cite{PD1} that a block circulant matrix of type $A(\ST_{m,n,k})$ can be block-diagonalized matrix by conjugating with the matrix $F_n \otimes F_m$, where $F_n = (f_{ij})_{n\times n}$ is the Fourier matrix is defined by $f_{ij} = \omega_n^{ij}$, $\omega_n$ being the primitive $n^{th}$ root of unity. A direct computation yields the following lemma. 
\begin{lemma}  
\label{lem:M_matrix}
For $0 \leq i,j \leq n-1$, let  $M= (M_{ij}) = (F_n \otimes F_m)A(\ST_{m,n,k})(F_n \otimes F_m)^*$.
Then: 
\begin{enumerate}[(i)]
\item for $1 \leq i \leq m-1$, \[
M_{ii} = \sum_{s=0}^{n-1} \omega^{is}_n F_m A_{0s} F^{\ast}_m,
\]
\item for $1 \leq i,j \leq m-1$, 
\[
(F_m A_{00} F^{\ast}_m)_{i,j} = m^{-1} \sum_{\tau=0}^{m-1} (\omega^{i-1}_m + {\overline{\omega}}^{i-1}_m) {\overline{\omega}}^{\tau(i-j)}_m, \text{ and}
\]
\item for $1 \leq \ell \leq n-1$ and $0 \leq i, j \leq m-1$, the matrix $F_m A_{0\ell} F^{\ast}_m$ is given by the following conditions.
\begin{enumerate}[(a)]
\item If $k^{s} \not\equiv -1$ (mod $n$) for any $s$, then we have 
$$(F_m A_{0\ell} F^{\ast}_m)_{ij} = 
\begin{cases}
\displaystyle 
\frac{1}{m}\sum_{s=0}^{t-1} {\overline{\omega}}^{ (\xi + s\alpha)(i - j) }_m,  &\mbox{if } \ell \equiv \pm k^{\alpha - \xi}, \text{ for } 0 \leq \xi \leq \alpha - 1, \text{ and} \\
0, &\mbox{otherwise.} 
\end{cases}$$
\item If $\alpha$ is even and $k^{\alpha/2} \equiv -1$ (mod $n$), then we have
\[
(F_m A_{0l} F^{\ast}_m)_{ij} = 
\begin{cases}
\displaystyle
\frac{1}{m}\sum_{s=0}^{2t-1} {\overline{\omega}}^{ (\xi + s\alpha/2)(i - j) }_m,   &\mbox{if } \ell \equiv \pm k^{\alpha - \xi}, \text{ for } 0 \leq \xi \leq (\alpha/2) - 1, \text{ and}\\
0, &\mbox{otherwise.} 
\end{cases}
\]
\end{enumerate}
\end{enumerate}
\end{lemma}
\noindent This leads us to the following proposition, which describes the elements of the diagonal blocks $M_{ii}$.

\begin{proposition} 
\label{prop:M_diagonal_blocks}
For $0 \leq i \leq n-1$ and $0 \leq a, b \leq m-1$, the matrix $M_{ii}$ is given by the following conditions. 
\begin{enumerate}[(i)]
\item If $k^{\tau} \not\equiv -1 \pmod{n}$, for some $\tau$, then  
\[ 
(M_{ii})_{a,b} =
\begin{cases}
\displaystyle
{\frac {t}{m}} \sum_{\xi=0}^{\alpha-1} (\omega^{ik^{\alpha - \xi}}_n + {\overline{\omega}}^{ik^{\alpha - \xi}}_n) {\overline{\omega}}^{\xi(a - b) }_m, &\mbox{if } t \mid a-b, a \neq b, \\
\displaystyle
(\omega^{a-1}_m + {\overline{\omega}}^{a-1}_m) + {\frac {t}{m}} \sum_{\xi=0}^{\alpha-1} (\omega^{ik^{\alpha - \xi}}_n + {\overline{\omega}}^{ik^{\alpha - \xi}}_n), &\mbox{if } a = b, \text{ and} \\
0, &\mbox{otherwise.} 
\end{cases}
\]
\item If $\alpha$ is even and $k^{\alpha/2} \equiv -1 \pmod{n}$, then 
\[ 
(M_{ii})_{a,b} = 
\begin{cases}
\displaystyle
{\frac {2t}{m}} \sum_{\xi=0}^{(\alpha/2)-1} (\omega^{ik^{\alpha - \xi}}_n + {\overline{\omega}}^{ik^{\alpha - \xi}}_n) {\overline{\omega}}^{\xi(a - b) }_m, &\mbox{if } 2t \mid a-b, a\neq b, \\
\displaystyle
(\omega^{a-1}_m + {\overline{\omega}}^{a-1}_m) + {\frac {2t}{m}} \sum_{\xi=0}^{(\alpha/2)-1} (\omega^{ik^{\alpha - \xi}}_n + {\overline{\omega}}^{ik^{\alpha - \xi}}_n), &\mbox{if } a = b, \text{ and}\\
0, &\mbox{otherwise.} 
\end{cases}
\]
\end{enumerate}
\end{proposition}
\begin{proof}
Be Lemma~\ref{lem:M_matrix} (i), we have
\[
M_{ii} = \sum_{s=0}^{n-1} \omega^{is}_n F_m A_{0s} F^{\ast}_m.
\]
Suppose that $k^{\tau} \not\equiv -1 \pmod{n}$, for some $\tau$. Then
\[
M_{ii} = F_m A_{00} F^{\ast}_m + \sum_{\xi = 0}^{\alpha-1} (\omega^{ik^{\alpha - \xi}}_n + \omega^{-ik^{\alpha - \xi}}_n) (F_m A_{0, k^{\alpha - \xi}} F^{\ast}_m).
\]
Thus, by Lemma~\ref{lem:M_matrix} (ii) - (iii), for $0 \leq a, b \leq m-1$, we have
\begin{eqnarray*}
(M_{ii})_{ab} & = & (F_m A_{00} F^{\ast}_m)_{ab} + \sum_{\xi = 0}^{\alpha-1} (\omega^{ik^{\alpha - \xi}}_n + {\overline{\omega}}^{ik^{\alpha - \xi}}_n) (F_m A_{0, k^{\alpha - \xi}} F^{\ast}_m)_{ab} \\
& = & \frac{1}{m}\sum_{\tau=0}^{m-1} (\omega^{a-1}_m + {\overline{\omega}}^{a-1}_m) {\overline{\omega}}^{\tau(a-b)}_m +  \frac{1}{m}\sum_{\xi = 0}^{\alpha-1} (\omega^{ik^{\alpha - \xi}}_n + {\overline{\omega}}^{ik^{\alpha - \xi}}_n) \sum_{s=0}^{t-1} {\overline{\omega}}^{ (\xi + s \alpha)(a - b) }_m.
\end{eqnarray*}
Replacing $\tau$ by $\xi + s \alpha$,  we get
\begin{eqnarray*}
(M_{ii})_{a,b} & = &  \frac{1}{m}\sum_{\xi=0}^{\alpha-1} \sum_{s=0}^{t-1} {\overline{\omega}}^{ (\xi + s \alpha)(a - b) }_m \bigl[(\omega^{a-1}_m + {\overline{\omega}}^{a-1}_m) + (\omega^{ik^{\alpha - \xi}}_n + {\overline{\omega}}^{ik^{\alpha - \xi}}_n) \bigr] \\
& = &  \frac{1}{m} \sum_{\xi=0}^{\alpha-1} \bigl[(\omega^{a-1}_m + {\overline{\omega}}^{a-1}_m) + (\omega^{ik^{\alpha - \xi}}_n + {\overline{\omega}}^{ik^{\alpha - \xi}}_n) \bigr] {\overline{\omega}}^{\xi(a - b) }_m \sum_{s=0}^{t-1} ({\overline{\omega}}^{\alpha(a - b) }_m)^s.
\end{eqnarray*}
As ${\overline{\omega}}^{\alpha(a - b) }_m$ is a $t^{th}$ root of unity, we have
\[
\sum_{s=0}^{t-1} ({\overline{\omega}}^{\alpha(a - b) }_m)^s = 
\begin{cases}
0  &\mbox{if } t \nmid a-b \\
t &\mbox{otherwise,} 
\end{cases}
\]
which implies that
\[
(M_{ii})_{a,b} = tm^{-1} \sum_{\xi=0}^{\alpha-1} \bigl[(\omega^{a-1}_m + {\overline{\omega}}^{a-1}_m) + (\omega^{ik^{\alpha - \xi}}_n + {\overline{\omega}}^{ik^{\alpha - \xi}}_n) \bigr] {\overline{\omega}}^{\xi(a - b) }_m,
\]
if $t \mid a-b$, and is $0$ otherwise. When $t \mid a-b$, ${\overline{\omega}}^{a-b}$ is an $\alpha^{th}$ root of unity, and so
\[
\sum_{\xi=0}^{\alpha-1} {\overline{\omega}}^{\xi(a - b) }_m = 
\begin{cases}
\alpha, &\mbox{if } a = b, \text{ and} \\
0 &\mbox{otherwise,}
\end{cases}
\] 
which proves (i). The result in (ii) follows from a similar argument. 
\end{proof}

\noindent An immediate consequence of Proposition~\ref{prop:M_diagonal_blocks} is the following corollary.

\begin{corollary}
\label{cor:m_spectrum}
The block $M_{00} = (m_{ij})_{m\times m}$ of $M$ is a diagonal matrix given by $m_{kk} = 2 + 2\cos(2 \pi k/m)$, for $0 \leq k \leq m-1$. 
\end{corollary}

\section{Spectral expanders in $\{\ST_{m,n,k}\}$}
Since solvable groups of bounded derived length do not yield expander families of graphs, the family $\{\ST_{m,n,k}\}$ as a whole is not expander. However, it is an interesting pursuit to determine the graphs in $\{\ST_{m,n,k}\}$ with the largest possible spectral gaps (i.e. Ramanujan graphs), as these graphs are expected to be the best spectral expanders. To fix notation,  let $$k = \lambda_0(X) \geq \lambda_1(X) \geq \ldots \geq \lambda_{n-1}(X) \geq -k$$ be the spectrum of a $k$-regular graph $X$, and let 
$$\lambda(X) = \begin{cases}
\max\{\lambda_1(X)|, |\lambda_{n-2}(X)|\}, & \text{if } X \text{ is bipartite, and}\\
\max\{\lambda_1(X)|, |\lambda_{n-1}(X)|\}, & \text{otherwise.}
\end{cases}.$$  
Such a graph $X$ is said to be \textit{Ramanujan}, if $\lambda(X) \leq 2\sqrt{k-1}$. Clearly, a necessary condition condition for $\ST_{m,n,k}$ to be Ramanujan $\lambda(\ST_{m,n,k}) \leq 2\sqrt{3}$. From Corollary~\ref{cor:m_spectrum}, it is apparent that $\lambda(\ST_{m,n,k}) \geq 2 + 2 \cos(2\pi/m)$, which implies that $\lambda(\ST_{m,n,k}) > 2\sqrt{3}$, when $m > 8$, and so we have the following corollary. 
\begin{corollary}
The graph $\ST_{m,n,k}$ is not Ramanujan when $m > 8$.
\end{corollary}

 \noindent  For a fixed $2 \leq m \leq 8$, we can find significantly large $n$ yielding Ramanujan $(m,n,k)$-supertoroids. For example, our computations show that $\ST_{8, 388,33}$ is Ramanujan. So, in order to derive an upper bound on $n$, we will require a positive lower bound for $\lambda(\ST_{m,n,k})$ that involves $n$. 
 
 \subsection{A lower bound for $\lambda(\ST_{m,n,k})$} In order to derive such a bound, we will use the following extension of a result due to Walker-Mieghem~\cite{WM1}, which gives a lower bound for the largest eigenvalue of a Hermitian matrix. Though the proof of this result is analogous to the proof in~\cite{WM1}, we include it for completion.

\begin{theorem} 
\label{lem:lb_largest_evalue}
Let $A = (a_{ij})_{m \times m}$ be a Hermitian matrix, and let $\displaystyle \tilde{\lambda} = \sup_{0 \leq i \leq m-1}|\lambda_i(A)|$. Suppose that $f(t) = \sum_{k=0}^{\infty} f_k t^{k}$ is an analytic function with radius of convergence $R_f$ such that $f_k$ is real for each $k$, and $f(t)$ is increasing for all real $t$. Then for all $t > \tilde{\lambda}/R_f$, the largest eigenvalue of $A$ can be bounded from below by
\[
\displaystyle tf^{-1} \left ( \frac{1}{m} \sum_{k=0}^{\infty} f_k N_k t^{-k} \right ),
\]
where $N_k = u^T A^k u = \sum_{i=1}^{m} \sum_{j=1}^{m} (A^k)_{ij}$ and $N_0 = m$, $u = [1, 1, \dotsc, 1]^T$ is an $m \times 1$ vector.
\end{theorem}

\begin{proof}
 Let $A = [a_{ij}]_{1\leq i,j \leq m}$ be an $m \times m$ Hermitian matrix with eigenvalues 
\[
\lambda_{\mathrm{min}}(A) = \lambda_1 \leq \lambda_2 \leq \dotsc \leq \lambda_m = \lambda_{\mathrm{max}}(A).
\]
Using Rayleigh's quotient theorem (see \cite[Theorem 4.2.2]{HJ}), we have 
\[
\lambda_{\mathrm{max}}(A) = \max_{x \in {\mathbb C}^m \setminus \{ 0 \}} {\frac {x^{T} A x}{ x^T x}}.
\]  
Setting $x = u = [1, \dotsc, 1]^T$, we have 
\[
\lambda_{\mathrm{max}}(A) \geq {\frac {1}{m}} \sum_{i, j =1}^{m} a_{ij}.
\]
Also, note that $x^{T} A x \in {\mathbb R}$, for any $x \in {\mathbb C}^m$, which follows from the spetral theorem for Hermitian matrices. Now, consider a complex valued function $f(z)$ with radius of convergence $R_f > 0$ around $0$ and have the Taylor series expansion
\[
f(z) = \sum_{k=0}^{\infty} a_k z^k,
\] 
which satisfies: (i) $a_k \in {\mathbb R}$ for all $k \in {\mathbb N}_0$, (ii) $f$ is increasing on the real line. We now apply this to the series
\[
A_t = \sum_{k=0}^{\infty} a_k A^k t^{-k},
\]
where $t > 0$ real. Note that in general a series $\sum_{k=0}^{\infty} b_k A^k$ is convergent if, and only if, $||A|| < R_g$ for any matrix norm $||.||$, where $g(z) = \sum_{k=0}^{\infty} b_k z^k$ is a Taylor expansion with radius of convergence $R_g > 0$ around $0$ (see \cite[Theorem 5.6.15]{HJ}). Thus, to make sense of $A_t$ we require $||A||/t < R_f$, that is, $t > ||A||/{R_f}$. This is possible by choosing $t > 0$ sufficiently large. 

Suppose that for such a $t > 0$, for an eigenvalue $\lambda (\in {\mathbb R})$ of $A$ corresponding to an eigenvector $v \in {\mathbb C}^m$, we have 
\[
A_t v = \sum_{k=0}^{\infty} a_k A^k t^{-k} v = \sum_{k=0}^{\infty} a_k \lambda^k t^{-k} v = f \Bigl({\frac {\lambda}{t}} \Bigr) v,
\]
where to make sense of convergence of the series we need to ensure $|\lambda/t| < R_f$. This can be done by ensuring $t > {\lambda_0}/{R_f}$, where ${\lambda_0} = \max_{1 \leq j \leq m} |\lambda_j|$.

 Since $f$ is increasing, using spectral mapping theorem for normal operators it follows that for $t R_f > \max \bigl\{ ||A||, \lambda_0 \bigr\}$, and so we have 
\begin{equation}
\label{eqn:wm-1}
\lambda_{\mathrm{max}} (A_t) = f \Bigl( {\frac {\lambda_{\mathrm{max}}(A)}{t}} \Bigr). 
\end{equation}
For $t > 0$ sufficiently large, applying Rayleigh's classical bound to $A_t$, we get
\begin{equation}
\label{eqn:wm-2}
{\lambda_{\mathrm {max}}} (A_t) \geq {\frac {u^T A_t u}{m}} = {\frac {1}{m}} \sum_{k=1}^{\infty} f_k N_k t^{-k}, 
\end{equation} 
where $N_k = u^T A^k u$, for each $k \geq 0$. Moreover, applying Rayleigh's bound to $A^k$ we have $N_k \leq m {\lambda_{\mathrm{max}}}(A^k)$. Since ${\lambda_{\mathrm{max}}}(A^k) \leq \lambda^k_0$, we have $N_k \leq m \lambda^k_0$, and hence the series on the right of (\ref{eqn:wm-1}) converges, if $\lambda_0/t < R_f$, that is, for $t > \lambda_0/{R_f}$. Finally, since $f^{-1}(x)$ is increasing on the real line, the assertion follows from (\ref{eqn:wm-2}).
\end{proof}

In particular, we consider the case when $f(x) = e^{ax}$ (see~\cite[Section 2.1]{WM1}) for deriving our bound on $n$, for which the lower bound in Theorem~\ref{lem:lb_largest_evalue} takes the form 
\begin{equation}
\label{eqn:lower_bound}
\frac{N_1}{m} + \frac{1}{2} \left( \frac{N_2}{m} - \frac{N_1^2}{m^2} \right) \frac{a}{t} + \left\lbrace \frac{a}{3} \left( \frac{N_3}{m} - \frac{N_1^3}{m^3} \right ) + \left( \frac{N_1^3}{m^3} - \frac{N_1 N_2}{m^2} \right) \right\rbrace
\frac{a^2}{2t^2} + O(t^{-3}),
\end{equation}
$\text{ for } t > a\widetilde{T}\sqrt{m} /\log(2), \text{ where } \widetilde{T}\sqrt{m} = \max_{1 \leq j \leq m} \left \lbrace \sum_{i=1}^m |a_{ij}| \right \rbrace.$ In our case, we take $\widetilde{T}\sqrt{m} = 4$, as this is an upper bound for the absolute values of the eigenvalues of the $M_{ii}$. For notational consistency, we shall denote the parameters associated with our blocks $M_{ii}$, for $1 \leq i \leq n-1$, by $N^{(i)}_1, \, N^{(i)}_2$, and $N^{(i)}_3$. To further simplify notation, we define $k$ to be \textit{regular} if $k^s \equiv -1 \pmod{n}$, for all $s$, and we define $k$ to be \textit{irregular},  if $2 \mid \alpha$ and $k^{\alpha/2} \equiv -1 \pmod{n}$. Furthermore, we fix the notation, 
\[
\small 
\Omega^n_{i, a, b} =
\begin{cases}
{\frac{t}{m}} \sum_{\xi=0}^{\alpha-1} (\omega^{ik^{\alpha - \xi}}_n + {\overline{\omega}}^{ik^{\alpha - \xi}}_n) {\overline{\omega}}^{\xi(a - b) }_m,  &\mbox{if k is regular} \text{ and } t \mid a-b, \\ 
{\frac {2t}{m}} \sum_{\xi=0}^{(\alpha/2)-1} (\omega^{ik^{\alpha - \xi}}_n + {\overline{\omega}}^{ik^{\alpha - \xi}}_n) {\overline{\omega}}^{\xi(a - b) }_m, &\mbox{if k if irregular and } 2t \mid a-b, \text{ and} \\
0, &\mbox{otherwise.}
\end{cases}
\]
When $a = b$, as $\Omega^n_{i, a, b}$ is independent of $a$, we simply denote it by $\Omega_i^n$. We will also need the following two technical lemmas that involve several tedious computations.
\begin{lemma}
\label{lem:Omega_n}
With $\Omega^n_{i, a, b}$ and $\Omega_i^n$  defined as above, we have: 
\begin{enumerate}[(i)]
\item $\displaystyle \sum_{a=0}^{m-1} (M_{ii})_{a, a} = m \Omega^n_i.$ 
\item $\displaystyle \frac {1}{m} \sum_{0 \leq a, b \leq m-1} (\omega^{a-1}_m + {\overline{\omega}}^{a-1}_m)(\omega^{b-1}_m + {\overline{\omega}}^{b-1}_m) \Omega^n_{i, a, b} = 
\begin{cases}
\omega^{ik}_n + {\overline {\omega}}^{ik}_n + \omega^{ik^{\alpha - 1}}_n + {\overline {\omega}}^{ik^{\alpha - 1}}_n, &\mbox{if k is regular},  \\
\omega^{ik^{\alpha - 1}}_n + {\overline {\omega}}^{ik^{\alpha - 1}}_n, &\mbox{if k is irregular.} 
\end{cases}$
\item Let $\delta = 1$, if $k$ is regular, and $\delta =2$, otherwise. Then the following sums vanish
\begin{enumerate}
\item $\displaystyle \sum_{0 \leq a \neq b \leq m-1, \delta t \mid a-b} \Omega^n_{i, a, b} \omega^{a-1}_m, \sum_{0 \leq a \neq b \leq m-1, \delta t \mid a-b} \Omega^n_{i, a, b} {\overline{\omega}}^{a-1}_m,$
\item $\displaystyle \sum_{0 \leq a \neq b \leq m-1, \delta t \mid a-b} \Omega^n_{i, a, b} \omega^{b-1}_m, \sum_{0 \leq a \neq b \leq m-1, \delta t \mid a-b} \Omega^n_{i, a, b} {\overline{\omega}}^{b-1}_m, \text{ and}$
\item  $\displaystyle  {\frac {1}{m}} \sum_{0 \leq a, b \leq m-1} \Bigl[ \omega^{b-1}_m \bigl( \Omega^n_{i(1 + k^{-1}), a, b} + \Omega^n_{i(1 - k^{-1}), a, b} \bigr) + {\overline{\omega}}^{b-1}_m \bigl( \Omega^n_{i(1 + k), a, b} + \Omega^n_{i(1 - k), a, b} \bigr) \Bigr]$, and
\item $\displaystyle \frac {1}{m} \sum_{0 \leq a \neq b \leq m-1, \delta t \mid a-b} \Omega^n_{i, a, b} (\omega^{a-1}_m + {\overline{\omega}}^{a-1}_m + \omega^{b-1}_m + {\overline{\omega}}^{b-1}_m).$
\end{enumerate}
\end{enumerate}
\end{lemma}

\begin{proof}
From Proposition~\ref{prop:M_diagonal_blocks}, we have
$$(M_{ii})_{a,b} = \Omega_{i,a,b}^n + \tau_{i,a,b}^m,$$ where 
\[
\tau^m_{i, a, b} =
\begin{cases}
(\omega^{a-1}_m + {\overline{\omega}}^{a-1}_m), &\mbox{if } a = b, \text{ and} \\ 
0, &\mbox{otherwise.}
\end{cases}
\] 
Hence, we have 
\[
\sum_{a=0}^{m-1} (M_{ii})_{a, a} = \sum_{a=0}^{m-1} (\omega^{a-1}_m + {\overline{\omega}}^{a-1}_m) + m \Omega^n_i.
\]
Since $\omega_m$ is an $m^{th}$  root of unity, the first sum vanishes, and (i) follows.

To show (ii), we first consider the case when $k$ is regular. Upon reparametrizing the entries of the sum as
\[
(a, a + \eta t) ~:~ 0 \leq a \leq m-1, 0 \leq \eta \leq \alpha - 1,
\]
and rearranging, the sum on the left becomes
\begin{equation}
\label{eqn:sum_ii}
{\frac {1}{m}} \sum_{\eta = 0}^{\alpha - 1} \Omega^n_{i, a, a + \eta t} \sum_{a=0}^{m-1} (\omega^{a-1}_m + {\overline{\omega}}^{a-1}_m) (\omega^{a + \eta t -1}_m + {\overline{\omega}}^{a + \eta t -1}_m).
\end{equation}
Since $m \geq 3$, the inner sum above can be rewritten as
\[
 \omega^{\eta t}_ m \sum_{a=0}^{m-1} \omega^{2(a-1)}_m + {\overline {\omega}}^{\eta t}_ m \sum_{a=0}^{m-1} {\overline {\omega}}^{2(a-1)}_m + m(\omega^{\eta t}_ m + {\overline {\omega}}^{\eta t}_ m) = m(\omega^{\eta t}_ m + {\overline {\omega}}^{\eta t}_ m)
\]
Consequently, after rearranging, (\ref{eqn:sum_ii}) yields
\[
\sum_{\eta = 0}^{\alpha - 1} \Omega^n_{i, a, a + \eta t} (\omega^{\eta t}_m + {\overline{\omega}}^{\eta t}_m) = {\frac {t}{m}} \sum_{\xi = 0}^{\alpha - 1} (\omega^{ik^{\alpha - \xi}}_n + {\overline{\omega}}^{ik^{\alpha - \xi}}_n) \sum_{\eta = 0}^{\alpha - 1} (\omega^{(\xi + 1) \eta t}_m + \omega^{(\xi - 1) \eta t}_m).
\]
The inner sum on the right can be further broken into two parts, where the two parts survive if, and only if, $\xi = \alpha - 1$ and $\xi = 1$, respectively. Further, when these parts are nonzero, each assumes the value $\alpha$, and hence the result follows for the case when $k$ is regular.  The argument for the case when $k$ is irregular is similar. 

Consider the first sum in (iii)(a). While $k$ is regular, we subdivide this sum into parts with a fixed $a$, and allowing $b$ to run over all elements $a + \eta t$ (mod $m$) $(1 \leq \eta \leq \alpha-1)$. As the factor $\Omega^n_{i, a, a + \eta t}$ depends only on $\eta$, the sum equals
\[
\sum_{a = 0}^{m-1} \sum_{\eta = 1}^{\alpha-1} \Omega^n_{i, a, a + \eta t} \omega^{a-1}_m =  \sum_{\eta = 1}^{\alpha-1} \Omega^n_{i, a, a + \eta t} \sum_{a = 0}^{m-1} \omega^{a-1}_m = 0.
\]
On the other hand, when $k$ is irregular, the sum runs over all $b = a + \eta (2t)$ (mod $m$) $(1 \leq \eta \leq {\frac {\alpha}{2}}-1)$. The arguments for the other sum in (iii)(a) and the sums in (iii)(b) are analogous. 

For the sum in (iii) (c), we first note that by (iii) (a) - (b), the terms of the sum are nonzero, only when $\delta t \mid a-b$. Thus, it suffices to verify that
\[
{\frac {1}{m}} \sum_{a = 0}^{m-1} \omega^{a-1}_m \bigl( \Omega^n_{i(1 + k^{-1})} + \Omega^n_{i(1 - k^{-1})} \bigr) = 0 = {\frac {1}{m}} \sum_{a = 0}^{m-1} {\overline{\omega}}^{a-1}_m \bigl( \Omega^n_{i(1 + k)} + \Omega^n_{i(1 - k)} \bigr).
\]  
But this follows from the fact that the terms within the parantheses in the two sums above are independent of $a$. 
\end{proof}

\begin{lemma}
\label{lem:Mii_powers}
For $0 \leq a,b \leq m-1$, let $M^2_{ii} = (u_{ab})$ and $M^3_{ii} = (v_{ab})$, and let $\delta = 1$, if $k$ is regular, and $\delta =2$, otherwise. Then: 
\begin{enumerate}[(i)]
\item $u_{ab} =
\begin{cases}
\Omega^n_{i, a, b} (\omega^{a-1}_m + {\overline{\omega}}^{a-1}_m + \omega^{b-1}_m + {\overline{\omega}}^{b-1}_m) + \Omega^n_{2i, a, b}, &\mbox{if } a \neq b, \delta t \mid a-b \\
(\omega^{a-1}_m + {\overline{\omega}}^{a-1}_m)^2 + 2 + 2 \Omega^n_i (\omega^{a-1}_m + {\overline{\omega}}^{a-1}_m) + \Omega^n_{2i}, &\mbox{if } a = b, \text{ and } \\
0, &\mbox{otherwise.}
\end{cases}$
\item $\small
v_{ab} = 
\begin{dcases}
\Omega^n_{i, a, b} \Bigl[ (\omega^{a-1}_m + {\overline{\omega}}^{a-1}_m)^2 + (\omega^{b-1}_m + {\overline{\omega}}^{b-1}_m)^2 + (\omega^{a-1}_m + {\overline{\omega}}^{a-1}_m)(\omega^{b-1}_m + {\overline{\omega}}^{b-1}_m) + 3 \Bigr] \\
+ \Omega^n_{2i, a, b} \Bigl[ (\omega^{a-1}_m + {\overline{\omega}}^{a-1}_m) + (\omega^{b-1}_m + {\overline{\omega}}^{b-1}_m) \Bigr] + \Omega^n_{3i, a, b} \\
+ \omega^{b-1}_m \bigl( \Omega^n_{i(1 + k^{-1}), a, b} + \Omega^n_{i(1 - k^{-1}), a, b} \bigr) + {\overline{\omega}}^{b-1}_m \bigl( \Omega^n_{i(1 + k), a, b} + \Omega^n_{i(1 - k), a, b} \bigr),   
&\mbox{if } a \neq b \text{ and }\delta t \mid a - b,\\
(\omega^{a-1}_m + {\overline{\omega}}^{a-1}_m)^3 + 4(\omega^{a-1}_m + {\overline{\omega}}^{a-1}_m) + 3 \Omega^n_i \Bigl[ (\omega^{a-1}_m + {\overline{\omega}}^{a-1}_m)^2 + 1 \Bigr] \\
+ 2 \Omega^n_{2i} (\omega^{a-1}_m + {\overline{\omega}}^{a-1}_m) + \Omega^n_{3i} \\
+ \omega^{a-1}_m \bigl( \Omega^n_{i(1 + k^{-1})} + \Omega^n_{i(1 - k^{-1})} \bigr) + {\overline{\omega}}^{a-1}_m \bigl( \Omega^n_{i(1 + k)} + \Omega^n_{i(1 - k)} \bigr),
&\mbox{if } a = b, \text{ and}\\
0, &\mbox{if } \delta t \nmid a-b.
\end{dcases}$
\end{enumerate}
\end{lemma}
\begin{proof}
By definition, we have 
$u_{ab} = \sum_{\tau=0}^{m-1} (M_{ii})_{a, \tau} (M_{ii})_{\tau, b}$. First, we consider the case when $k$ is regular. If $t \nmid a-b$, clearly, by definition the $u_{ab} = 0$, since either $(M_{ii})_{a, \tau} = 0$, or $(M_{ii})_{\tau, b} = 0$. When $a \neq b$, from the definition we have $u_{ab}$ equals
\begin{equation}
\label{eqn:uab}
\Omega^n_{i, a, b} (\omega^{a-1}_m + {\overline{\omega}}^{a-1}_m + \omega^{b-1}_m + {\overline{\omega}}^{b-1}_m) + \sum_{\tau=0}^{m-1} \Omega^n_{i, a, \tau} \Omega^n_{i, \tau, b}.
\end{equation}
In the expression for $u_{ab}$ in (\ref{eqn:uab}) above, the terms $\Omega^n_{i, a, b}, \, \Omega^n_{i, a, \tau}, \text{ and }\Omega^n_{i, \tau, b}$ are nonzero, only when $t \mid a-\tau, \tau-b$. Now the last term in (\ref{eqn:uab}) is by definition
\[
{\frac {t^2}{m^2}} \sum_{\xi=0}^{\alpha-1} \sum_{\xi^{\prime}=0}^{\alpha-1} (\omega^{ik^{\alpha - \xi}}_n + {\overline{\omega}}^{ik^{\alpha - \xi}}_n) (\omega^{ik^{\alpha - \xi^{\prime}}}_n + {\overline{\omega}}^{ik^{\alpha - \xi^{\prime}}}_n) \sum_{0 \leq \tau \leq m-1, t \mid a- \tau, \tau-b} {\overline{\omega}}^{\xi(a - \tau) + \xi^{\prime}(\tau-b)}_m,
\]
where the inner sum runs on the indices $\tau = a, a+t, a+2t, \dotsc, a+(\alpha-1)t \pmod{m}$, and so the inner sum is
\[
\sum_{\eta=0}^{\alpha-1} {\overline {\omega}}^{(-\xi) \eta t + {\xi}^{\prime} [(a-b) + \eta t]}_m = {\overline {\omega}}^{\xi^{\prime} (a-b)}_m \sum_{\eta=0}^{\alpha-1} ({\overline {\omega}}^t_m)^{(\xi^{\prime} - \xi)\eta}.
\]
Since ${\overline {\omega}}^t_m$ is an $\alpha^{th}$ root of unity and $\alpha \mid \xi^{\prime} - \xi$ if, and only if,  $\xi^{\prime} = \xi$ , the inner sum is nonzero, only when $\xi^{\prime} = \xi$, and while it is nonzero, it is equal to $\alpha {\overline {\omega}}^{\xi(a-b)}_m$. Thus, the last sum in (\ref{eqn:uab}) reduces to
\[
{\frac {t}{m}} \sum_{\xi=0}^{\alpha-1} (\omega^{ik^{\alpha - \xi}}_n + {\overline{\omega}}^{ik^{\alpha - \xi}}_n)^2 {\overline {\omega}}^{\xi(a - b)}_m
= {\frac {t}{m}} \sum_{\xi=0}^{\alpha-1} (\omega^{2ik^{\alpha - \xi}}_n + {\overline{\omega}}^{2ik^{\alpha - \xi}}_n) {\overline {\omega}}^{\xi(a - b)}_m + {\frac {2t}{m}} \sum_{\xi=0}^{\alpha-1} {\overline {\omega}}^{\xi(a - b)}_m,
\]
from which the result follows for the case $a \neq b$. When $a = b$, $u_{ab}$ equals
\begin{equation}
\label{eqn:uab_case2}
(\omega^{a-1}_m + {\overline{\omega}}^{a-1}_m)^2 + 2 \Omega^n_i (\omega^{a-1}_m + {\overline{\omega}}^{a-1}_m) + \sum_{0 \leq \tau \leq m-1, t \mid a-\tau} \Omega^n_{i, a, \tau} \Omega^n_{i, \tau, a}.
\end{equation}
As before, the last sum in (\ref{eqn:uab_case2}) above equals
\[
{\frac {t^2}{m^2}} \sum_{\xi=0}^{\alpha-1} \sum_{\xi^{\prime}=0}^{\alpha-1} (\omega^{ik^{\alpha - \xi}}_n + {\overline{\omega}}^{ik^{\alpha - \xi}}_n) (\omega^{ik^{\alpha - \xi^{\prime}}}_n + {\overline{\omega}}^{ik^{\alpha - \xi^{\prime}}}_n) \sum_{0 \leq \tau \leq m-1, t \mid a- \tau} {\overline{\omega}}^{\xi(a - \tau) + \xi^{\prime}(\tau-a)}_m,
\]
where the inner sum equals
\[
\sum_{\eta=0}^{\alpha-1} ({\overline {\omega}}^t_m)^{(\xi^{\prime} - \xi)\eta}.
\]
Thus, the last sum  (\ref{eqn:uab_case2}) reduces to
\[
{\frac {t}{m}} \sum_{\xi=0}^{\alpha-1} (\omega^{ik^{\alpha - \xi}}_n + {\overline{\omega}}^{ik^{\alpha - \xi}}_n)^2 = \Omega^n_{2i} + 2,
\]
and the result follows for this case. The argument for the case when $k$ is irregular is analogous. 

For showing (ii), we will first establish the following claim: \\

\noindent \textbf{Claim.}  For $0 \leq a, b \leq m-1$, the sum 
\[
\sum_{\tau = 0}^{m-1} (\omega^{\tau-1}_m + {\overline {\omega}}^{\tau-1}_m) \Omega^n_{i, a, \tau} \Omega^n_{i, \tau, b}
\]
is given by
\[\small
=
\begin{cases}
\omega^{b-1}_m \bigl( \Omega^n_{i(1 + k^{-1}), a, b} + \Omega^n_{i(1 - k^{-1}), a, b} \bigr) + {\overline{\omega}}^{b-1}_m \bigl( \Omega^n_{i(1 + k), a, b} + \Omega^n_{i(1 - k), a, b} \bigr), &\mbox{if } \delta t \mid a - b,  \alpha = 4, \text{ and }\\
&k \text{ is regular},  \\
(\omega^{b-1}_m + {\overline{\omega}}^{b-1}_m) \bigl( \Omega^n_{i(1 + k), a, b} + \Omega^n_{i(1 - k), a, b} \bigr), &\mbox{if } 2t \mid a - b, \alpha = 4, \text{ and}\\
 &k \text{ is irregular, and}\\
0 &\mbox{if } \delta t \nmid a - b.
\end{cases}
\]
\noindent \textit{Proof of claim.} While $\delta t \nmid a - b$, the sum is clearly $0$. We first consider the case when $a \neq b$ and $k$ is regular, and so the sum becomes
\begin{equation}
\label{eqn:claim_sum}
 {\frac {t^2}{m^2}} \sum_{\xi=0}^{\alpha-1} \sum_{\xi^{\prime}=0}^{\alpha-1} (\omega^{ik^{\alpha - \xi}}_n + {\overline{\omega}}^{ik^{\alpha - \xi}}_n) (\omega^{ik^{\alpha - \xi^{\prime}}}_n + {\overline{\omega}}^{ik^{\alpha - \xi^{\prime}}}_n) \sum_{0 \leq \tau \leq m-1, t \mid a- \tau} (\omega^{\tau-1}_m + {\overline {\omega}}^{\tau-1}_m) {\overline{\omega}}^{\xi(a - \tau) + \xi^{\prime}(\tau-b)}_m.
\end{equation}
As this runs over all indices
\[
\tau = a + \eta t ~({\mathrm {mod}}~m),~ \eta = 0, 1, \dotsc, \alpha -1,
\]
the inner sum equals
\begin{equation}
\label{eqn:vab_claim}
 {\overline {\omega}}^{\xi^{\prime} (a-b)}_m \sum_{\eta = 0}^{\alpha -1} (\omega^{a + \eta t -1}_m + {\overline {\omega}}^{a + \eta t-1}_m) ({\overline {\omega}}_m^t)^{(\xi^{\prime} - \xi)\eta}.
\end{equation}
Breaking (\ref{eqn:vab_claim}) further into two parts, we obtain
\[
{\overline {\omega}}^{\xi^{\prime} (a-b)}_m \omega^{a-1}_m \sum_{\eta=0}^{\alpha-1} {\overline{\omega}}_m^{(\xi^{\prime} - \xi - 1) \eta t} + {\overline {\omega}}^{\xi^{\prime} (a-b)}_m {\overline{\omega}}^{a-1}_m \sum_{\eta=0}^{\alpha-1} {\overline{\omega}}_m^{(\xi^{\prime} - \xi + 1) \eta t}.
\]
Note that first part is nonzero, only when $\xi^{\prime} = \xi + 1$, in which case, it takes the value $\alpha {\overline {\omega}}^{(\xi + 1)(a-b)}_m \omega^{a-1}_m$, while the second part survives, only when $\xi^{\prime} = \xi - 1$, where it assumes the value $\alpha {\overline {\omega}}^{(\xi - 1)(a-b)}_m {\overline{\omega}}^{a-1}_m$. As $\alpha \geq 3$, the indices $\xi + 1 \text{ and }\xi - 1$ are distinct modulo $\alpha$. Plugging these values in (\ref{eqn:claim_sum}), we get
\[
 {\frac {t}{m}} {\overline {\omega}}^{a-b}_m \omega^{a-1}_m \sum_{\xi = 0}^{\alpha - 1} (\omega^{ik^{\alpha - \xi}}_n + {\overline{\omega}}^{ik^{\alpha - \xi}}_n) (\omega^{ik^{\alpha - \xi - 1}}_n + {\overline{\omega}}^{ik^{\alpha - \xi - 1}}_n) {\overline {\omega}}^{\xi (a-b)}_m
\]
\[
+ {\frac {t}{m}} \omega^{a-b}_m {\overline{\omega}}^{a-1}_m \sum_{\xi = 0}^{\alpha - 1} (\omega^{ik^{\alpha - \xi}}_n + {\overline{\omega}}^{ik^{\alpha - \xi}}_n) (\omega^{ik^{\alpha - \xi + 1}}_n + {\overline{\omega}}^{ik^{\alpha - \xi + 1}}_n) {\overline {\omega}}^{\xi (a-b)}_m,
\]
from which the result follows for this case. The arguments for the case when $k$ is irregular and the case $a = b$ is similar, and so the claim follows.
We will now proceed to prove (ii).  When $a \neq b$, we have
\begin{eqnarray}
\label{eqn:vab_expr}
v_{ab} & = & \sum_{\tau = 0}^{m-1} (M^2_{ii})_{a, \tau} (M_{ii})_{\tau, b} \nonumber \\
& = & (M^2_{ii})_{a, a} (M_{ii})_{a, b} + (M^2_{ii})_{a, b} (M_{ii})_{b, b} + \sum_{\tau \neq a, b} (M^2_{ii})_{a, \tau} (M_{ii})_{\tau, b} \nonumber \\
& = & \bigl[ (\omega^{a-1}_m + {\overline{\omega}}^{a-1}_m)^2 + 2 + 2 \Omega^n_i (\omega^{a-1}_m + {\overline{\omega}}^{a-1}_m) + \Omega^n_{2i} \bigr] \Omega^n_{i, a, b} \nonumber \\
& &+ \bigl[ \Omega^n_{i, a, b} (\omega^{a-1}_m + {\overline{\omega}}^{a-1}_m + \omega^{b-1}_m + {\overline{\omega}}^{b-1}_m) + \Omega^n_{2i, a, b} \bigr] (\omega^{b-1}_m + {\overline{\omega}}^{b-1}_m + \Omega^n_i)  \nonumber \\
& & + \sum_{\tau \neq a, b} \bigl[ \Omega^n_{i, a, \tau} (\omega^{a-1}_m + {\overline{\omega}}^{a-1}_m + \omega^{\tau-1}_m + {\overline{\omega}}^{\tau-1}_m) + \Omega^n_{2i, a, \tau} \bigr] \Omega^n_{i, \tau, b}
 \nonumber \\
& = & \sum_{\tau = 0}^{m-1} (\omega^{\tau-1}_m + {\overline {\omega}}^{\tau-1}_m) \Omega^n_{i, a, \tau} \Omega^n_{i, \tau, b} + (\omega^{a-1}_m + {\overline{\omega}}^{a-1}_m) \sum_{\tau = 0}^{m-1} \Omega^n_{i, a, \tau} \Omega^n_{i, \tau, b} + \sum_{\tau = 0}^{m-1} \Omega^n_{2i, a, \tau} \Omega^n_{i, \tau, b}\nonumber  \\
& &+ \bigl[ (\omega^{a-1}_m + {\overline{\omega}}^{a-1}_m)^2 + (\omega^{b-1}_m + {\overline{\omega}}^{b-1}_m)^2 + (\omega^{a-1}_m + {\overline{\omega}}^{a-1}_m)(\omega^{b-1}_m + {\overline{\omega}}^{b-1}_m) + 2 \bigr] \Omega^n_{i, a, b} \nonumber \\
& & + (\omega^{b-1}_m + {\overline{\omega}}^{b-1}_m) \Omega^n_{2i, a, b}. 
\end{eqnarray}
From (\ref{eqn:vab_expr}) above, it is clear that $v_{ab} = 0$, if $\delta t \nmid a-b$. Observing that the third sum is $\Omega^n_{3i, a, b} + \Omega^n_{i, a, b}$,  we further simplify (\ref{eqn:vab_expr})  using the claim we just proved and the proof of (i),  to obtain the desired result for the case $a \neq b$. The argument for the case when $a = b$ is analogous. 
\end{proof}

\noindent We will now derive an expression for $N_k^i$, for $1 \leq k \leq 3$.

\begin{proposition} 
\label{prop:Nki}
For $1 \leq i \leq n-1$, we have:
\begin{enumerate}[(i)]
\item $\displaystyle \frac {N^{(i)}_1}{m} = \omega^i_n + {\overline {\omega}}^i_n$, 
\item $\displaystyle \frac {N^{(i)}_2}{m} = 4 + \omega^{2i}_n + {\overline {\omega}}^{2i}_n$, and 
\item $\displaystyle \frac{N^{(i)}_3}{m} =
\begin{cases}
\omega^{ik}_n + {\overline {\omega}}^{ik}_n + \omega^{ik^{\alpha - 1}}_n + {\overline {\omega}}^{ik^{\alpha - 1}}_n + 7(\omega^{i}_n + {\overline {\omega}}^{i}_n) + (\omega^{3i}_n + {\overline {\omega}}^{3i}_n), &\text{if k is regular, and} \\
\omega^{ik^{\alpha - 1}}_n + {\overline {\omega}}^{ik^{\alpha - 1}}_n + 7(\omega^{i}_n + {\overline {\omega}}^{i}_n) + (\omega^{3i}_n + {\overline {\omega}}^{3i}_n), &\mbox{if k is irregular.}
\end{cases}$

\end{enumerate}
\end{proposition}
\begin{proof}
By applying Lemma~\ref{lem:Omega_n} (i), we get
\begin{eqnarray*}
{\frac {N^{(i)}_1}{m}} & = &{\frac {1}{m}} \sum_{0 \leq a, b \leq m-1} (M_{ii})_{a, b} \\
& = & {\frac {1}{m}} \sum_{a=0}^{m-1} (M_{ii})_{a, a} + {\frac {1}{m}} \sum_{0 \leq a \neq b \leq m-1} (M_{ii})_{a, b} \\
& = & \Omega^n_i + {\frac {1}{m}} \sum_{0 \leq a \neq b \leq m-1} \Omega^n_{i, a, b}  \\
& = &{\frac {1}{m}} \sum_{0 \leq a, b \leq m-1} \Omega^n_{i, a, b}. 
\end{eqnarray*}
Suppose that $k$ is regular. Then 
\begin{eqnarray*} 
{\frac {N^{(i)}_1}{m}} & = & {\frac {1}{m}} \sum_{0 \leq a, b \leq m-1,~ t \mid a-b} {\frac {t}{m}} \sum_{\xi = 0}^{\alpha-1} (\omega^{ik^{\alpha - \xi}}_n + {\overline{\omega}}^{ik^{\alpha - \xi}}_n) {\overline{\omega}}^{\xi(a - b) }_m \\
& = & {\frac {t}{m^2}} \sum_{\xi = 0}^{\alpha-1} (\omega^{ik^{\alpha - \xi}}_n + {\overline{\omega}}^{ik^{\alpha - \xi}}_n) \sum_{0 \leq a, b \leq m-1,~ t \mid a-b} {\overline{\omega}}^{\xi(a - b) }_m.
\end{eqnarray*} 
For each $0 \leq a \leq m-1$, setting $b = a + \eta t$, for $0 \leq \eta \leq \alpha -1 $, we get
\begin{eqnarray*}
{\frac {N^{(i)}_1}{m}} & = & {\frac {t}{m^2}} \sum_{\xi = 0}^{\alpha-1} (\omega^{ik^{\alpha - \xi}}_n + {\overline{\omega}}^{ik^{\alpha - \xi}}_n) \sum_{a = 0}^{m-1} \sum_{\eta = 0}^{\alpha - 1} {\overline {\omega}}^{\xi (- \eta t)}_m \\ 
& = & {\frac {t}{m}} \sum_{\xi = 0}^{\alpha-1} (\omega^{ik^{\alpha - \xi}}_n + {\overline{\omega}}^{ik^{\alpha - \xi}}_n) \sum_{\eta = 0}^{\alpha - 1} (\omega^t_m)^{\xi \eta}.
\end{eqnarray*}
Since ${\overline {\omega}}^t_m$ is an $\alpha^{th}$ root of unity, and so the last sum is $0$ when $\xi \neq 0$, and is $\alpha$, when $\xi = 0$, from which (i) follows. The proof for the case when $k$ is irregular follows from similar arguments. 

To show (ii),  we first apply Lemma~\ref{lem:Mii_powers} (i),  to obtain
\begin{eqnarray*}
{\frac {N^{(i)}_2}{m}} & = & {\frac {1}{m}} \sum_{a=0}^{m-1} \Bigl[ (\omega^{a-1}_m + {\overline{\omega}}^{a-1}_m)^2 + 2 + 2 \Omega^n_i (\omega^{a-1}_m + {\overline{\omega}}^{a-1}_m) + \Omega^n_{2i} \Bigr]  \\
& & + {\frac {1}{m}} \sum_{0 \leq a \neq b \leq m-1, \delta t \mid a-b} \Bigl[ \Omega^n_{i, a, b} (\omega^{a-1}_m + {\overline{\omega}}^{a-1}_m + \omega^{b-1}_m + {\overline{\omega}}^{b-1}_m) + \Omega^n_{2i, a, b} \Bigr] \\
& = & {\frac {1}{m}} \sum_{a=0}^{m-1} \Bigl[ (\omega^{2a-2}_m + {\overline{\omega}}^{2a-2}_m) + 4 +  \Omega^n_{2i} \Bigr] + {\frac {1}{m}} \sum_{0 \leq a \neq b \leq m-1, \delta t \mid a-b} \Omega^n_{2i, a, b}.
\end{eqnarray*}
If $m$ is odd, both $\omega^{2a-2}_m$ and ${\overline{\omega}}^{2a-2}_m$ vary over all powers of $\omega_m$, and so we have
\[
{\frac {1}{m}} \sum_{a=0}^{m-1} (\omega^{2a-2}_m + {\overline{\omega}}^{2a-2}_m) = 0
\]
Similarly, when $m$ is even we have $\omega^2_m$ is a primitive $(m/2)^{th}$ root of unity, and hence
\[
{\frac {1}{m}} \sum_{a=0}^{m-1} \omega^{2a-2}_m = {\frac {1}{m}} \sum_{a=0}^{{\frac{m}{2}}-1} \omega^{2a-2}_m + (\omega^2_m)^{m/2} {\frac {1}{m}} \sum_{b=0}^{{\frac{m}{2}}-1} \omega^{2b-2}_m = 0.
\]
Therefore, we have 
\[
{\frac {N^{(i)}_2}{m}} = 4 + {\frac {1}{m}} \sum_{a=0}^{m-1} \Omega^n_{2i} + {\frac {1}{m}} \sum_{0 \leq a \neq b \leq m-1, \delta t \mid a-b} \Omega^n_{2i, a, b} = 4 + {\frac {N^{(2i)}_1}{m}}, \text{ and}
\]
(ii) follows.

To show (iii), first suppose that $\alpha \geq 6$ and $k$ is irregular. Using Lemma~\ref{lem:Mii_powers} (ii) and (i), we have
\begin{eqnarray*}
{\frac {N^{(i)}_3}{m}} & = & \sum_{a,b = 0}^{m-1} \Omega^n_{i, a, b} \Bigl[ (\omega^{a-1}_m + {\overline{\omega}}^{a-1}_m)^2 + (\omega^{b-1}_m + {\overline{\omega}}^{b-1}_m)^2 \Bigr] \\
& & + \sum_{a,b = 0}^{m-1} \Omega^n_{i, a, b} \Bigl[ (\omega^{a-1}_m + {\overline{\omega}}^{a-1}_m)(\omega^{b-1}_m + {\overline{\omega}}^{b-1}_m) \Bigr] +\\
& & + 3 \sum_{a,b = 0}^{m-1} \Omega^n_{i, a, b} + \sum_{a,b = 0}^{m-1} \Omega^n_{2i, a, b} \Bigl[ (\omega^{a-1}_m {\overline{\omega}}^{a-1}_m) + (\omega^{b-1}_m + {\overline{\omega}}^{b-1}_m) \Bigr] \\
& & + \sum_{a,b = 0}^{m-1} \Omega^n_{3i, a, b} +\sum_{a=0}^{m-1} \Bigl[ (\omega^{a-1}_m + {\overline{\omega}}^{a-1}_m)^3 + 4(\omega^{a-1}_m + {\overline{\omega}}^{a-1}_m) \Bigr].
\end{eqnarray*}
By applying (i) and Lemma~\ref{lem:Omega_n}, and by observing that the first sum is $4N^{(i)}_1/m$, the third sum is $3N^{(i)}_1/m$, the fifth sum is $N^{(3i)}_1/m$, the result follows.
\end{proof}

\noindent Plugging in the expressions for $N_k^{(i)}$ from Proposition~\ref{prop:Nki} in Equation~(\ref{eqn:lower_bound}), we obtain the following. 

\begin{corollary}
\label{cor:lambdaX_bound}
The largest eigenvalue of $M_{ii}$ is bounded below by
$$B(n, k, a, t, \epsilon) := (\omega_n + {\overline {\omega}}_n) + {\frac {a}{t}} + \Bigl( {\frac {\epsilon a}{3}} -2 \Bigr) (\omega_n + {\overline {\omega}}_n) \Bigl( {\frac {a^2}{2t^2}} \Bigr) + O(t^{-3}),$$ where $\epsilon = 2$, if $k$ is regular, and $\epsilon = 3$, otherwise. Consequently, $\lambda(X) \geq B(n, k, a, t, \epsilon)$.
\end{corollary}

\noindent It finally remains to obtain a lower bound for $n$ beyond which the graph $\ST_{m,n,k}$ will not be Ramanujan.

\subsection{A lower bound for n} The aim of this subsection is to establish that
 $\ST_{m,n,k}$ is not Ramanujan when $n \geq 400$. To begin with, we calculate the sum of the first three terms of $B(n, k, a, t, \epsilon)$, by choosing $a = 127.5$, $\epsilon =2 $, $t = 1000$, and $n \geq 400$, which turns out to be $\geq 3.476$ (i.e $\geq 2\sqrt{3}$). Note that when $\epsilon = 3$, choosing $a = 111.5$, and the same values of $t$ and $n$ as above, would bound the sum of the non-error terms of  $B(n, k, a, t, \epsilon)$ from below by $3.472$. To obtain estimates for the error terms in the bound $B(n, k, a, t, \epsilon)$ in Corollary~\ref{cor:lambdaX_bound}, we use the general form (see ~\cite[Appendix A]{WM1}) of the Taylor expansion of $f^{-1} ({\frac {1}{m}} \sum_{k=0}^{\infty} f_k N_k z^k)$ around $z_0 = 0$, which is given by
\[
f^{-1} (h(z)) = \sum_{k=1}^{\infty} c_k z^k
\] 
where $h(z) = {\frac {1}{m}} \sum_{k=0}^{\infty} f_k N_k z^k$ and the \textit{characteristic coefficients}
\begin{equation}
\label{eqn:ck}
\begin{split}
c_k = {\frac {f_k}{f_1}} \bigg( {\frac {N_k}{m}} - \Bigl( {\frac {N_1}{m}} \Bigr)^k \bigg) + {\frac {1}{k}} \Bigl( {\frac {N_1}{m}} \Bigr)^k \sum_{j=2}^{k-1} (-1)^j {k+j-1 \choose j} f_1^{-j} s^{\ast} [j, k-1] \lvert_{f(z)} \\
+ \sum_{n=2}^{k-1} \Bigg( \sum_{j=1}^{n-1} (-1)^j {n+j-1 \choose j} f^{-n-j}_1 s^{\ast} [j, n-1] \lvert_{f(z)} \Bigg) {\frac {s[n, k] \lvert_{h(z)}}{n}}.
\end{split}
\end{equation}

We define $$M(n, t) := \begin{cases} 
                                                 \displaystyle \sum_{n_1 + \dotsc + n_t = n, \, n_i \geq 2} {n \choose n_1 ~ n_2 \dotsc n_t}, & \text{if } t \geq 2, \text{ and} \\
                                                 1, & \text{if } t=1, 
                                                 \end{cases}$$
where ${n \choose n_1 ~ n_2 \dotsc n_t} = {\frac {n!}{n_1! \,n_2! \dotsc n_t!}}$. Notice that the quantity $M(n+j, j)$ is the number of surjective functions from a set of size $n+j$ onto a set of size $j$. Further, we set
$$S_n := {\frac {1}{n!}} \sum_{j=2}^{n} {\frac {(-1)^j}{j!}} M(n+j, j),$$ and the state a lemma, which we will use in our error estimates. 
\begin{lemma}
\label{lem:Sn_no_proof}
For $n \geq 2$, we have
\[
S_n := {\frac {1}{n!}} \sum_{j=2}^{n} {\frac {(-1)^j}{j!}} M(n+j, j) = (-1)^n + {\frac {1}{n!}}.
\]
\end{lemma}

\begin{proof}
Multiplying both sides by $n!$, we need to prove that
\[
A(n) := \sum_{j=1}^{n} {\frac {(-1)^j}{j!}} M(n+j, j) = (-1)^n n!.
\]
Denoting $I_n = \{ 1, 2, \dotsc, n \}$, we first note that $f(n,j) := {\frac {M(n+j, j; 2)}{j!}}$ denotes the number of partitions of $I_{n+j}$ into $j$ subsets, each of which contains at least two elements. For any such partition of $I_{n+j}$, if the element $n+j$ belong to a subset (of this partition) of size $\geq 3$, then removing it will yield a partition of $I_{n+j-1}$ into $j$ subsets of the above type. Labeling the subsets by indices $1, 2, \dotsc, j$, this can be done in $j$ distinct ways. 

If the element $n+j$ belongs to a subset of size $2$, then removing this subset yields a partition of the set $I_{n+j-1} \setminus \{ x \}$ into $j-1$ subsets, for any $x \in I_{n+j-1}$. This can be done in $n+j-1$ ways. Counting each of these partitions separately we get a recursive formula
\[
f(n,j) = j f(n-1,j) + (n+j-1) f(n-1, j-1).
\]
Plugging this in the above sum for $A(n)$, we obtain
\[
A(n) = \sum_{j=1}^{n} (-1)^j j f(n-1, j) + \sum_{j=1}^{n} (-1)^j (n+j-1) f(n-1, j-1).
\]
Now changing the variable $j \mapsto j+1$ in the second sum, combining the coefficients of $f(n-1,k)$ together, and using $f(n-1,n) = 0$, we obtain $A(n) = -n A(n-1)$. The assertion now follows by induction on $n$.
\end{proof}

We denote the three terms (in the order of their appearance) in the expression for $c_k$ in (\ref{eqn:ck}) for the case $f(x) = e^{ax}$ by $R_{i,k}$, for $1 \leq i \leq 3$, where 
\begin{equation}
\label{eqn:R1k}
R_{1,k} = {\frac {a^{k-1}}{k!}} \bigg( {\frac {N_k}{m}} - \Bigl( {\frac {N_1}{m}} \Bigr)^k \bigg).
\end{equation}
We will now derive simpler expressions for $R_{2,k}$ and $R_{3,k}$.

\begin{proposition}
\label{prop:R23k}
\begin{enumerate}[(i)]
\item $\displaystyle R_{2,k} = {\frac {a^{k-1}}{k}} \Bigl( {\frac {N_1}{m}} \Bigr)^k \left(  (-1)^n + \frac {1}{n!}\right)$. 
\item $\displaystyle R_{3,k} = a^{k-1} \sum_{n=2}^{k-1} \frac{(-1)^{n-1}}{n}\sum_{ \sum_{i=1}^{n} t_i = k, ~t_i \geq 1} \bigl( {\frac {N_{t_1}}{m}} \bigr) \dotsc \bigl( {\frac {N_{t_n}}{m}} \bigr).$
\end{enumerate}
\end{proposition}
\begin{proof} 
The term $s^{\ast} [j, k-1] \lvert_{f(z)}$ of $T_2$ is given by
\begin{eqnarray*}
s^{\ast} [j, k-1] \lvert_{f(z_0)} & = & \sum_{ \sum_{i=1}^{j} t_i = k-1, t_i \geq 1} \prod_{i=1}^{j} f_{t_i + 1} (z_0) \\
& = & \sum_{ \sum_{i=1}^{j} t_i = k-1, t_i \geq 1} {\frac {a^{k-1 + j}}{ (t_1 +1)! \dotsc (t_j + 1)! }} \\
& = & {\frac {a^{k-1+j}}{ (k-1+j)! }} \sum_{ \sum_{i=1}^{j} t_i = k-1, t_i \geq 1} {\frac { (k-1+j)! }{ (t_1 +1)! \dotsc (t_j + 1)! }} \\ & = & {\frac {a^{k-1+j}}{ (k-1+j)! }} M(k-1+j, j).
\end{eqnarray*}

Consequently, we have 
\begin{eqnarray*}
R_{2,k} & = & {\frac {1}{k}} \Bigl( {\frac {N_1}{m}} \Bigr)^k \sum_{j=2}^{k-1} (-1)^j {k+j-1 \choose j} f_1^{-j} s^{\ast} [j, k-1] \lvert_{f(z_0)} \\
& = & {\frac {1}{k}} \Bigl( {\frac {N_1}{m}} \Bigr)^k \sum_{j=2}^{k-1} (-1)^j {k+j-1 \choose j} {\frac {a^{k-1}}{ (k-1+j)! }} M(k-1+j, j) \\
& = & {\frac {a^{k-1}}{k!}} \Bigl( {\frac {N_1}{m}} \Bigr)^k \sum_{j=2}^{k-1} (-1)^j {\frac {M(k-1+j, j; 2)}{j!}},
\end{eqnarray*}
from which (i) follows by a direct application of Lemma~\ref{lem:Sn_no_proof}.

To show (ii), we first compute $s[n, k] \lvert_{h(z_0)}$. We have
\begin{eqnarray}
\label{eqn:snk}
s[n, k] \lvert_{h(z_0)} & = & \sum_{ \sum_{i=1}^{n} t_i = k, ~t_i \geq 1} \prod_{i=1}^{n} h_{t_i} (z_0) \nonumber \\
& = & \sum_{ \sum_{i=1}^{n} t_i = k, ~t_i \geq 1} a^{t_1} \bigl( {\frac {N_{t_1}}{m}} \bigr) \dotsc a^{t_n} \bigl( {\frac {N_{t_n}}{m}} \bigr) \nonumber \\
& = & a^k \sum_{ \sum_{i=1}^{n} t_i = k, ~t_i \geq 1} \bigl( {\frac {N_{t_1}}{m}} \bigr) \dotsc \bigl( {\frac {N_{t_n}}{m}} \bigr)
\end{eqnarray}
From the computations for $R_{2,k}$, we see that
\begin{equation}
\label{eqn:innersum}
\small \sum_{j=1}^{n-1} (-1)^j {n+j-1 \choose j} f^{-n-j}_1 s^{\ast} [j, n-1] \lvert_{f(z)} =  {\frac {1}{a (n-1)!}} \sum_{j=1}^{n-1} (-1)^j {\frac {M(n-1+j, j)}{j!}}.
\end{equation}
Finally, (ii) follows from~(\ref{eqn:snk}),~(\ref{eqn:innersum}), and Lemma~\ref{lem:Sn_no_proof}.
\end{proof}

It now remains to obtain upper bounds on the error terms $R_{i,k}$, whenever $n \geq 400$.  We can obtain upper bounds for $R_{1,k}$ and $R_{2,k}$ by plugging in the inequalities  $N_k/m \leq 4^k$ and $N_1/m = 2\cos(2\pi/n) \geq (1.9997533)^k$ in~(\ref{eqn:R1k}). Moreover, we can obtain an upper bound for $R_{3,k}$ by applying the inequalities
$$\sum_{ \sum_{i=1}^{n} t_i = k, ~t_i \geq 1} \bigl( {\frac {N_{t_1}}{m}} \bigr) \dotsc \bigl( {\frac {N_{t_n}}{m}} \bigr)  \leq 4^{2k-1} \text{ and } \left |\sum_{n=2}^{k-1} \frac{(-1)^{n-1}} {n}\right | \leq \log\,2 + 1$$ to Proposition~\ref{prop:R23k}, where the second inequality follows from the fact number of integer solutions of the equation $t_1+\ldots+t_n = k, \, t_i \geq 1$ is ${k+n-1 \choose k} \leq 4^{k-1}$. This leads us to the following corollary. 

\begin{corollary}
\label{cor:bounds_on_Rik}
For $n \geq 400$, we have: 
\begin{enumerate}[(i)]
\item $\displaystyle | R_{1,k} | \leq \frac{1}{ak!} {\left( \frac{a}{t} \right)}^k (4^k - 1.9997533^k),$
\item $\displaystyle | R_{2,k} | \leq \frac{1}{ak} {\left( \frac{a}{t} \right)}^k (1.9997533^k)\left | (-1)^n + \frac{1}{n!} \right |,$ and 
\item $\displaystyle | R_{3,k} | \leq  (\log \, 2 + 1) \frac{1}{a} {\left( \frac{a}{t} \right)}^k(1.9997533^k)4^{k - 1}$.
\end{enumerate}
\end{corollary}
\noindent For $4 \leq k \leq 10$, the following table lists the values of the bounds on the $R_{i,k}$ (which we denote by $S_{i,k}$) from Corollary~\ref{cor:bounds_on_Rik},  assuming that $a = 127.5$, $t = 1000$, and $n \geq 400$. (It is apparent that under these assumptions, $|S_{i,k}| \to 0$, as $k \to \infty$.)

\begin{center}
\begin{tabular}{|c|c|c|c|}
\hline
 $k$ & $S_{1,k}$ & $S_{2,k}$ & $S_{3,k}$  \\ 
 \hline 
4 & $2.072740039 \times 10^{-5}$ & $1.325855621 \times 10^{-4}$ & $3.591789931 \times 10^{-3}$ \\  
5 & $2.184639608 \times 10^{-6}$  & $4.225643495 \times 10^{-5}$ & $3.663173821 \times 10^{-3}$\\  
6 & $1.886879001 \times 10^{-7}$ & $1.292887412 \times 10^{-5}$ & $3.735976408 \times 10^{-3}$\\  
7 & $1.385629798 \times 10^{-8}$ &  $3.845865605 \times 10^{-6}$ & $3.810225887 \times 10^{-3}$\\  
8 & $8.868141122 \times 10^{-10}$ & $1.120656869 \times 10^{-6}$ &  $3.885951013 \times 10^{-3}$\\  
9 & $5.035124916 \times 10^{-11}$  &  $3.214487837 \times 10^{-7}$ & $3.963181115 \times 10^{-3}$ \\  
10 & $2.570423859 \times 10^{-12}$ & $9.106592102 \times 10^{-8}$ & $4.041946102 \times 10^{-3}$\\
 \hline

\end{tabular}
 \label{tab:1}
\end{center}

\noindent A similar set of values can be computed for the case when $\epsilon =3$. In conclusion, we have the following result, which follows directly from Corollaries~\ref{cor:lambdaX_bound} and~\ref{cor:bounds_on_Rik}, and the discussion at the beginning of this subsection. 
\begin{theorem}
The graph $\ST_{m,n,k}$ is not Ramanujan when $n \geq 400$. 
\end{theorem}

\noindent We conclude this paper by enlisting a collection of triples $(m,n,k)$ such that given any $(m,n)  \in \{(x,y) : 3 \leq x \leq 8 \text{ and } 3 \leq y < 400\}$ there exist at least one nontrivial unit $k \in \mathbb{Z}_n ^\times$ such that $\ST_{m,n,k}$ is Ramanujan.

\noindent {\small (3, 7, 2), (3, 9, 4), (3, 13, 3), (3, 14, 9), (3, 18, 7), (3, 19, 7), (3, 21, 4), (3, 26, 3), (3, 28, 9), (3, 31, 5), (3, 35, 11), (3, 37, 10), (3, 38, 7), (3, 39, 16), (3, 42, 25), (3, 52, 9), (3, 56, 9), (3, 62, 5), (3, 63, 25), (3, 74, 47), (3, 78, 55), (3, 117, 16), (4, 5, 2), (4, 8, 3), (4, 10, 3), (4, 12, 5), (4, 13, 5), (4, 15, 2), (4, 16, 3), (4, 20, 3), (4, 24, 5), (4, 26, 5), (4, 30, 7), (4, 32, 7), (4, 39, 5), (4, 40, 3), (4, 48, 5), (4, 52, 5), (4, 60, 7), (4, 78, 5), (4, 80, 7), (4, 104, 5), (4, 120, 7), (5, 11, 3), (5, 25, 6), (5, 31, 2), (5, 33, 4), (5, 41, 10), (5, 55, 16), (5, 61, 9), (5, 71, 5), (5, 77, 15), (5, 93, 4), (5, 101, 36), (5, 121, 9), (5, 123, 10), (5, 131, 53), (5, 151, 8), (5, 155, 66), (5, 181, 59), (5, 183, 58), (5, 191, 39), (5, 217, 8), (5, 241, 87), (5, 251, 149), (5, 275, 141), (5, 311, 36), (5, 341, 70), (5, 363, 130), (6, 7, 2), 
(6, 8, 3), (6, 9, 2), (6, 12, 5), (6, 13, 3), (6, 14, 3), (6, 16, 7), (6, 18, 5), (6, 19, 7), (6, 21, 2), (6, 24, 5), (6, 26, 3), (6, 28, 3), (6, 31, 5), (6, 36, 5), (6, 37, 10), (6, 38, 7), (6, 39, 4), (6, 42, 5), (6,52, 3), (6, 56, 3), (6, 57, 11), (6, 62, 5), (6, 72, 5), (6, 74, 11), (6, 76, 7), (6, 78, 17), (6, 84, 5), (6, 91, 4), (6, 93, 5), (6, 104, 17), (6, 111, 26), (6, 112, 9),(6, 114, 11), (6, 117, 4), (6, 124, 37), (6, 133, 26), (6, 148, 47), (6, 152, 7), (6, 156, 17), (6, 168, 5), (6, 171, 68), (6, 182, 23), (6, 186, 5), (6, 208, 55), (6, 222, 47), (6, 234, 29), (6, 248, 37), (6, 266, 45), (6, 279, 88), (6, 296, 85), (7, 29, 7), (7, 43, 4), (7, 71, 30), (7, 87,7), (7, 113, 16), (7, 145, 16), (7, 197, 114), (7, 211, 58), (7, 213, 172), (7, 215, 176), (7, 239, 44), (7, 339, 16), (8, 5, 2), (8,8, 3), (8, 10, 3), (8, 12, 5), (8, 13, 5), (8, 15, 2), (8, 16, 3), (8, 17, 2), (8, 20, 3), (8, 24, 5), (8, 26, 5), (8, 30, 7), (8, 32, 3), (8, 34, 9), (8, 39, 5), (8, 40, 3), (8, 48, 5), (8, 51, 8), (8, 52, 5), (8, 60, 7), (8, 64, 7), (8, 68, 9), (8, 73, 10), (8, 78, 5), (8, 80, 7), (8, 85, 8), (8, 89, 12), (8, 96, 5), (8, 102, 19), (8, 104, 5), (8, 113, 18), (8, 120, 7), (8, 136, 9), (8, 146, 51), (8, 160, 13), (8, 170, 53), (8, 178, 37), (8, 194, 33), (8, 204, 19), (8, 219, 10), (8, 226, 69), (8, 272, 9), (8, 292, 83), (8,340, 93), (8, 356, 101), (8, 388, 33)}

\subsection{Acknowledgements} The authors would like to thank Ram Murty for a careful reading of an earlier version of this paper.  The authors would also like to thank Fedor Petrov for the proof of Lemma~\ref{lem:Sn_no_proof}.
\bibliographystyle{plain}
\bibliography{supertoroid_paper}
 \end{document}